\documentclass[11pt]{amsart}

\usepackage{amsmath}
\usepackage{amssymb}
\usepackage{graphicx}

\hoffset = - 0.5in

\newtheorem{theorem}{Theorem}[section]

\newtheorem{proposition}[theorem]{Proposition}
\theoremstyle{definition}
\newtheorem{definition}[theorem]{Definition}

\numberwithin{equation}{section}

\renewcommand{\geq}{\geqslant}
\renewcommand{\leq}{\leqslant}

\title{Property($K^*$) Implies $R(X) \leq 1 + \frac{\displaystyle 1}{\displaystyle 1 + r_{X^*}(1)}$}

\author[Tim Dalby]{Tim Dalby}

\date{\today}

\keywords{weak fixed point property, property($K^*$), Opial's condition, Opial's modulus}

\subjclass[2010]{46B10, 47H09, 47H10}

\email{tim\_dalby@bigpond.com}

\begin{document}

\parindent = 0pt
\parskip = 8pt

\begin{abstract}

It is shown that if the dual of a Banach space satisfies Property($K^*$) then $R(X) \leq 1 + \frac{\displaystyle 1}{\displaystyle 1 + r_{X^*}(1)} < 2$ where  $r_{X^*}(c)$ is Opial's modulus for $X^*.$ Thus $X$ has the weak fixed point property.

\end{abstract}

\maketitle

\section{Introduction}
   
The Banach space moduli of  $R(X), r_{X}(c)$ were introduced within the context of fixed point theory.  This area of research looks at whether, in an infinite dimensional real Banach space, every nonexpansive mappings on every weak compact convex nonempty set has a fixed point.  If this is so then the space is said to have the weak fixed point property (w-FPP).  It can be shown that this propery is separably determined, so in this paper the Banach space is assumed to be separable.  

If the sets are not necessarily weak compact but closed then the property is called the fixed point property (FPP).

Another aspect of this topic is that it involves weakly convergent sequences and usually a translation is employed so that the sequence being studied is weakly convergent to zero.  Inherent in most results is the interplay between weak null sequences and the norm.  Much more background can be found in Goebel and Kirk [7] and  Kirk and Sims [8].
   
The effect of the dual space, $X^*,$ on the outcomes in $X$ have also been studied.  See for example [3] and [4].  This paper continues along this line.

Definitions and notation will be presented in the next section.

In [3] Dalby showed that if $X$ is a separable Banach space where the dual, $X^*$, has Property($K^*$) and the nonstrict *Opial condition then $R(X) < 2$.  This last condition implies $X$ has the w-FPP.  For the latter result see [6].

In [1] Benavides showed that if $X$ is a reflexive Banach space and if there exist $c_0 \in (0, 1)$ where $r_{X^*}(c_0) > 0$ then $R(a, X) \leq \max \left \{ 1 + ac_0, a + \frac{\displaystyle 1}{\displaystyle 1 + r_{X^*}(c_0)} \right \}$ where $a \geq 0.$  This in turn implies $R(a, X) < 1 + a \mbox{ if } r_{X^*}(1) > 0.$  Here $r(c)$ is Opial's modulus. Earlier in that paper it was shown that $R(a, X) < 1 + a \mbox{ for some } a > 0$ ensures that $X$ has the w-FPP.  Because of reflexivity, in this case $X$ has the FPP.

Recall from [5] that $R(a, X) < 1 + a$ for some $a > 0$ is equivalent to $R(1, X) < 2$ and that $R(X) < 2$ implies $R(1, X) < 2.$

So the result in [1] can be rewritten as: If $X$ is a reflexive Banach space where $r_{X^*}(1) > 0$ then $X$ has the FPP.   But Dalby proved in [2] that $r_{X}(1) > 0$ is equivalent to $X$ having Property($K$).  This proof can readily be transferred to $X^*$.  That is $r_{X^*}(1) > 0$ is equivalent to $X^*$ having Property($K^*$).  Hence the result in [1] can be again rewritten as: If $X$ is a reflexive Banach space where $X^*$ has Property($K^*$) then $X$ has the FPP.

So the results of Benavides and Dalby are similar.

In this paper the requirement that $X$ be reflexive has been changed to $X$ being separable and the condition that $X^*$ has the nonstrict *Opial property has been removed.

\section{Definitions}

\begin{definition}

Sims, [11]

A Banach space $X$ has property($K$) if there exists $K \in [0, 1)$ such that whenever $x_n \rightharpoonup 0, \lim_{ n \rightarrow \infty} \| x_n \| = 1 \mbox{ and }  \liminf_{n \rightarrow \infty} \| x_n - x \|  \leq 1 \mbox{ then } \| x \| \leq K.$

If the sequence is in $B_{X^*}$ and is weak* convergent to zero then the property is called Property($K^*$).

\end{definition}

\begin {definition}

Opial [10] 

A Banach space has Opial's condition if
\[ x_n \rightharpoonup 0 \ \mbox {and } x \not = 0 \mbox { implies } \limsup_n \| x_n \| < \limsup_n \| x_n - x \|. \]
The condition remains the same if both the $\limsup$s are replaced by $\liminf$s.

If the $<$ is replaced by $\leq$ then the condition is called nonstrict Opial.

In $X^*$ the conditions are the same except that the sequence is weak* null.

\end {definition}

Later a modulus was introduced to gauge the strength of Opial's condition and a stronger version of the condition was defined.

\begin{definition}

Lin, Tan and Xu, [9]

Opial's modulus is
\[ r_X(c) = \inf \{ \liminf_{n\rightarrow \infty} \| x_n - x \| - 1: c \geq 0, \| x \| \geq c,  x_n \rightharpoonup 0 \mbox{ and }\liminf_{n\rightarrow \infty} \| x_n \| \geq 1 \}. \]
$X$ is said to have uniform Opial's condition if $r_X(c) > 0$ for all $c > 0.$  See [9] for more details.

\end{definition}

\section{Result}

\begin{proposition}

Let $X$ be a separable Banach space with $X^*$ having Property($K^*$).  Then $R(X) \leq 1 + \frac{\displaystyle 1}{\displaystyle 1 + r_{X^*}(1)}.$

\end{proposition}

\begin{proof}

Let $X$ be a separable Banach space such that $X^*$ has Property($K^*$).  Then $r_{X^*}(1) > 0.$  Also, let $x_n \in B_X \mbox{ for all } n , x_n \rightharpoonup 0 \mbox{ and } x \in B_X.$  Then for each $n \geq 1$ choose $x_n^* \in X^*$ such that $\| x_n^* \| = 1 \mbox{ and }x_n^*(x_n + x) = \| x_n + x \|$.  Using the property that $X$ is separable we can assume, without loss of generality, that $x_n^* \stackrel{*}{\rightharpoonup} x^* \mbox { where } \| x^* \| \leq 1$.

\begin{align*}
\liminf_{n \rightarrow \infty} \| x_n + x \| & = \liminf_{n \rightarrow \infty} x_n^*(x_n + x) \\
& = \liminf_{n \rightarrow \infty} x_n^*(x_n) + x^*(x) \\
& \leq \liminf_{n \rightarrow \infty} \| x_n^* \| \| x_n \| + \|x^* \| \| x \| \\
& = \liminf_{n \rightarrow \infty} \| x_n \| + \|x^* \| \| x \| \\
& \leq 1 + \| x^* \|. \qquad \qquad \dag
\end{align*}

The proof now splits into four cases.
 
{\bf Case 1 $ \bf x^* = 0$}

From \dag \hspace{1 pt} we have $\liminf_{n \rightarrow \infty} \| x_n + x \| \leq 1.$  This means that $R(X) \leq 1$ which in turn implies $R(X) = 1 \leq 1 + \frac{\displaystyle 1}{\displaystyle 1 + r_{X^*}(1)}.$

{\bf Case 2 $\bf x^* \not = 0, \| x^* \| \leq  \liminf_{n \rightarrow \infty} \| x_n^* - x^* \|$}

We have $\frac{\displaystyle x_n^*}{\displaystyle \| x^* \|} - \frac{\displaystyle x^*}{\displaystyle \| x^* \|} \stackrel{*}{\rightharpoonup} 0 \mbox{ and } \liminf_{n \rightarrow \infty} \left \| \frac{\displaystyle x_n^*}{\displaystyle \| x^* \|} - \frac{\displaystyle x^*}{\displaystyle \| x^* \|} \right \| \geq 1.$

Thus

\begin{align*} 
\liminf_{n \rightarrow \infty} \left \| \frac{x_n^*}{\| x^* \|} - \frac{x^*}{\| x^* \|} + \frac{x^*}{\| x^* \|}\right \| & \geq r_{X^*}(1) + 1 \\
1 = \lim_{n \rightarrow \infty} \| x_n^* \| & \geq \| x^* \| (r_X(1) + 1) \\
\| x^* \| & \leq \frac{1}{\displaystyle r_{X^*}(1) + 1}.
\end{align*}

From \dag \hspace{1pt} we have $\liminf_{n \rightarrow \infty} \| x_n + x \| \leq 1 + \frac{\displaystyle 1}{\displaystyle r_{X^*}(1) + 1}.$

Therefore $R(X) \leq 1 + \frac{\displaystyle 1}{\displaystyle r_{X^*}(1) + 1}.$

{\bf Case 3 $ \bf x^* \not = 0, 0 < \liminf_{n \rightarrow \infty} \| x_n^* - x^* \| < \| x^* \|$}

To make the presentation easier to read let $a = \liminf_{n \rightarrow \infty} \| x_n^* - x^* \|$ then \newline $a < \| x^* \| \leq 1.$

Note that $\frac{\displaystyle x_n^* - x^*}{\displaystyle a} \stackrel{*}{\rightharpoonup} 0 \mbox{ and } \liminf_{n \rightarrow \infty}\left \| \frac{\displaystyle x_n^* - x^*}{\displaystyle a} \right \| = 1.$

Now $\liminf_{n \rightarrow \infty}\left \| \frac{\displaystyle x_n^*}{\displaystyle a} - \frac{ \displaystyle x^*}{\displaystyle a} + \frac{\displaystyle  x^*}{\displaystyle a} \right \|  = \liminf_{n \rightarrow \infty} \frac{\|\displaystyle  x_n^* \|}{\displaystyle a} = \frac{\displaystyle 1}{\displaystyle a}.$

Using $\frac{\displaystyle \| x^* \|}{\displaystyle a} > 1, \liminf_{n \rightarrow \infty}\left \| \frac{\displaystyle x_n^*}{\displaystyle a} - \frac{\displaystyle  x^*}{\displaystyle a} + \frac{\displaystyle  x^*}{\displaystyle a} \right \|  \geq r_{X^*}(\frac{\displaystyle \| x^* \|}{\displaystyle a})  + 1 \geq  r_{X^*}(1) + 1$ since $r_{X^*}(c)$ is nondecreasing.

Thus $1 \geq a(r_{X^*}(1) + 1)$ and $ \frac{\displaystyle 1}{\displaystyle r_{X^*}(1) + 1} \geq a = \liminf_{n \rightarrow \infty} \| x_n^* - x^* \|.$

This leads to

\begin{align*}
\liminf_{n \rightarrow \infty} \| x_n + x \| & = \liminf_{n \rightarrow \infty} x_n^*( x_n + x ) \\
& = \liminf_{n \rightarrow \infty} (x_n^* - x^*)(x_n + x) + \liminf_{n \rightarrow \infty} x^*( x_n + x ) \\
& = \liminf_{n \rightarrow \infty} ( x_n^* - x^*)(x_n)+ x^*(x) \\
& \leq \liminf_{n \rightarrow \infty} \| x_n^* - x^* \| \| x_n \| + \| x \| \qquad \dag\dag \\
& \leq \frac{1}{r_{X^*}(1) + 1} + 1 \\
\mbox{ So } R(X) & \leq 1 + \frac{1}{r_{X^*}(1) + 1}.
\end{align*}

{\bf Case 4 $ \bf x^* \not = 0, \liminf_{n \rightarrow \infty} \| x_n^* - x^* \| = 0$} 

From \dag\dag, $\liminf_{n \rightarrow \infty} \| x_n + x \| \leq \| x \|$ but weak lower semicontinuity of the norm means $\| x \| \leq \liminf_{n \rightarrow \infty} \| x_n + x \|.$

Therefore $\liminf_{n \rightarrow \infty} \| x_n + x \| = \| x \| \leq 1\leq 1 + \frac{\displaystyle 1}{\displaystyle r_{X^*}(1) + 1}$ leading to \newline $R(X)  \leq 1 + \frac{\displaystyle 1}{\displaystyle r_{X^*}(1) + 1}.$
\end{proof}

{\bf Note } Careful reading of the proof reveals that $R(X) \leq 1 + \frac{\displaystyle 1}{\displaystyle 1 + r_{X^*}(1)}$ is {\bf always } true.

So Opial's modulus for $X^*$ plays a role in the behaviour of weak null sequences in $X.$

Property($K^*$)'s role is ensure that $r_{X^*}(1) > 0$ which in turn means that $R(X) < 2.$

{\bf Remark} Using $R(1, X) \leq RW(1, X) \leq R(X)$ we have:

Property($K^*$) implies $RW(1, X) \leq 1 + \frac{\displaystyle 1}{\displaystyle r_{X^*}(1) + 1}< 2$

and

Property( $K^*$) implies $R(1, X) \leq 1 + \frac{\displaystyle 1}{\displaystyle r_{X^*}(1) + 1} < 2.$ 

In [5] it was shown that $RW(1, X) < 2$ is equivalent to $RW(a, X) < 1 + a$ for some $a > 0.$  Similarly, $R(1, X) < 2$ is equivalent to $R(a, X) < 1 + a$ for some $a > 0.$


\begin{thebibliography}{99}

\bibitem {1} T. Dom\'{i}nguez-Benavides, {\it A geometric coefficient implying the fixed point property and stability results}, Houston J. Math. {\bf 22} (1996), 835-84

\bibitem {2} T. Dalby, {\it Relationships between properties that imply the weak fixed point property}, J. Math. Anal. Appl. {\bf 253} (2001), 578-589.

\bibitem {3} T. Dalby, {\it The effect of the dual on a Banach space and the weak fixed point property}, Bull. Austral. Math. Soc. {\bf 67} (2003), 177-185.

\bibitem {4} T. Dalby, {\it Property($K^*$) implies the weak fixed point property}, arXiv preprint arXiv:2007.00942 (2020).
 
\bibitem {5} T. Dalby, {\it Properties of $R(X), R(a,X) \mbox{ and } RW(a,X)$}, arXiv preprint arXiv:2005.08492 (2020).
 
\bibitem {6} J. Garc\'{i}a-Falset, {\it The fixed point property in Banach spaces with the NUS-property}, J. Math. Anal. Appl. {\bf 215} (1997), 532-534.
 
\bibitem {7} K. Goebel and W. A. Kirk, {\it Topics in metric fixed point theory}, Cambridge University Press, Cambridge, 1990.

\bibitem {8} W. A. Kirk and B. Sims (ed.), {\it Handbook of metric fixed point theory}, Kluwer Academic Publishers, Dordrecht,  2001.

\bibitem {9} P.-K. Lin, K.-K.Tan and H.-K. Xu, {\it Demiclosedness principle and asymptotic behavior for asymptotically nonexpansive mappings}, Nonlinear Anal. {\bf24} (1995), 929-946.

\bibitem{10} Z. Opial, {\it Weak convergence of the sequence of successive approximations for nonexpansive mappings}, Bull. Amer. Math. Soc. {\bf 73} (1967), 591-597.

\bibitem {11} B. Sims, {\it A class of spaces with weak normal structure}, Bull. Austral. Math. Soc. {\bf 49} (1994), 523-528.


\end{thebibliography}
\end{document}